\documentclass[11pt]{amsart}

 \usepackage[T1]{fontenc}
 \usepackage[utf8]{inputenc}
 \usepackage[english]{babel}
 \usepackage{graphicx}
 \usepackage{geometry}

\usepackage{amsmath,amsthm,amsfonts,amssymb}
\usepackage{mathtools}
\usepackage{enumerate}
\usepackage{enumitem}
 \usepackage{color}
 \usepackage{fancyhdr}
 \usepackage[titletoc]{appendix}
 \usepackage{tikz}
 \usepackage{tikz-cd}
 \usetikzlibrary{matrix,arrows,decorations.pathmorphing}

\newtheorem{theo}{Theorem}[section]
\newtheorem{coro}[theo]{Corollary}
\newtheorem{lemma}[theo]{Lemma}
\newtheorem{prop}[theo]{Proposition}
\theoremstyle{definition}

\theoremstyle{definition}
\newtheorem{rem}[theo]{Remark}

\input{xy}
\xyoption{all}
\xyoption{poly}
\usepackage[all]{xy}

 \newcommand\FF{{\mathbb{F}}}

 \newcommand\lk{{\mathrm{lk}}}
 \DeclareMathOperator{\im}{Im}
 
 \DeclareMathOperator{\ct}{ct}
 
 \DeclareMathOperator{\supp}{supp}
 
 \DeclareMathOperator{\Hom}{Hom}

\title {The covering type of closed surfaces and minimal triangulations}
   \author{Eugenio Borghini}
   \author{El\'ias Gabriel Minian}
   \address{Departamento  de Matem\'atica - IMAS\\
 FCEyN, Universidad de Buenos Aires. Buenos Aires, Argentina.}
\email{eborghini@dm.uba.ar ; gminian@dm.uba.ar}

\subjclass[2010]{57M20, 57Q15, 52B70, 57N16, 55M30, 55P15.}

\keywords{Covering type, minimal triangulations, surfaces.}

\begin{document}

   \begin{abstract}
  The notion of covering type was recently introduced by Karoubi and Weibel to measure the complexity of a topological space by means of good coverings. When $X$ has the homotopy type of a finite CW-complex, its covering type coincides with the minimum possible number of vertices of a simplicial complex homotopy equivalent to $X$. In this article we compute the covering type of all closed surfaces. Our results completely settle a problem posed by Karoubi and Weibel, and shed more light on the relationship between the topology of surfaces and the number of vertices of minimal triangulations from a homotopy point of view.
   \end{abstract}

   \maketitle

\section{Introduction}

It is well known that every smooth manifold admits a triangulation, but in general it is a hard problem to find minimal triangulations, i.e. triangulations with minimum number of vertices. In the case of surfaces, lower bounds on the number of vertices in terms of the Euler characteristic were known since the nineteenth century \cite{Hea} (see also \cite{Lut}). A complete solution to the problem of minimal triangulations of closed surfaces was given by Ringel \cite{Rin}, in the non-orientable case, and by Jungerman and Ringel \cite{JR}, in the orientable case. 

\begin{theo}\label{jungerman}
(Jungerman and Ringel) Let $S$ be a closed surface different from the orientable surface of genus $2$ ($M_2$), the Klein bottle ($N_2$) and the non-orientable surface of genus $3$ ($N_3$). There exists a triangulation of $S$ with $n$ vertices if and only if
\[
	n \geq \frac{7 + \sqrt{49-24\chi(S)}}{2}.
    \]
 For the exceptional cases $M_2, N_2$ and $N_3$, it is necessary to replace $n$ by $n - 1$ in the formula (see \cite{Hun,Lut}).
\end{theo}

Here $\chi(S)$ denotes the Euler characteristic of $S$. More generally, the problem of determining small triangulations of PL manifolds was extensively investigated by a number of authors (see for example \cite{BK,BK2,Lut,Lut2} and references therein).

It seems evident that the problem is much harder if one wants to determine the minimum number of vertices of triangulations up to homotopy type. This problem is closely related to the notion of covering type, which was recently introduced by Karoubi and Weibel \cite{KW}.

Recall that an open cover ${\mathcal U} = \{ U_i \}_{i \in I}$ of a topological space $X$ is a \textit{good cover} if every nonempty intersection $U_{i_1 } \cap \dots \cap U_{i_n}$ is contractible. The \textit{strict covering type} of $X$ is the minimum number of elements in a good cover of $X$. The \textit{covering type} of $X$, denoted by $\ct(X)$, is the minimum of the strict covering types taken over all spaces $X'$ homotopy equivalent to $X$. Good covers appear naturally in the context of Riemannian manifolds: any point in a Riemannian manifold has a geodesically convex neighborhood (which, in particular, is contractible) and the intersections of such neighborhoods are again geodesically convex. On one hand the notion of covering type is a refinement of the classical  concept of Lusternik-Schnirelmann category and, on the other hand, it is related to the number of vertices of minimal triangulations up to homotopy. Concretely,  if $X$ has the homotopy type of a finite CW-complex, $\ct(X)$ coincides with the minimum possible number of vertices of a simplicial complex homotopy equivalent to $X$ (see Lemma \ref{vertices} below).

In this paper we prove that for any closed surface $S$, its covering type coincides with the number of vertices in a minimal triangulation of $S$, with the exception of the orientable surface $M_2$. In particular our result asserts that, with the exception of $M_2$, a minimal triangulation of the homotopy type of $S$ can be attained in a triangulation of the surface itself. In the case of $M_2$, we prove that $\ct(M_2)=9$ (one less than the number of vertices of a minimal triangulation of $M_2$) and exhibit a simplicial complex homotopy equivalent to the surface with the minimum number of vertices of its homotopy type. 

Surprisingly the only tools required to prove these results are classical and basic: the Euler characteristic and the structure of the cohomology ring of the surface.

\section{Main result}

In this paper all the surfaces that we consider are closed and all the simplicial complexes are finite. Given a surface $S$, we denote by $\delta(S)$ the number of vertices in a minimal triangulation of $S$, and by $\rho(S)$ the minimum integer $n$ satisfying the inequality 
\[
	n \geq \frac{7 + \sqrt{49-24\chi(S)}}{2}
    \]
of Theorem \ref{jungerman}. Thus $\rho(S)$ coincides with $\delta(S)$ except for the exceptional cases $M_2, N_2, N_3$, in which cases we have $\rho(S) = \delta(S) - 1$.
 
In this section we prove the main result of the article: a simplicial complex $K$ homotopy equivalent to a surface $S$ has at least $\rho(S)$ vertices. This immediately implies that $\ct(S)=\delta(S)$ for the non-exceptional cases. The exceptional cases $M_2, N_2, N_3$ are studied separately in the last section of the paper.

We first describe the covering type in terms of minimum number of vertices.

\begin{lemma}\label{vertices}
Let $X$ be a topological space homotopy equivalent to a finite CW-complex. Then $\ct(X)$ coincides with the minimum possible number of vertices of a simplicial complex $K$ homotopy equivalent to $X$.
\end{lemma}

\begin{proof}
Suppose $\ct(X)=r$ and let $Y$ be a space homotopy equivalent to $X$ which admits a good cover ${\mathcal U}$ of size $r$. By the Nerve Theorem, $Y$ is homotopy equivalent to the nerve ${\mathcal N}({\mathcal U})$, which is a simplicial complex with $r$ vertices. On the other hand, if $K$ is a simplicial complex homotopy equivalent to $X$, the open stars of its vertices form a good cover of $K$.
\end{proof}

We will prove first a special case of the main theorem. Concretely, we will show that any simplicial complex of dimension $2$ with the Euler characteristic of a surface $S$ and whose cohomology ring satisfies certain regularity property, has at least $\rho(S)$ vertices. First we need a lemma. Recall that a simplicial complex of dimension $n$ is \textit{homogeneous} or \textit{pure} if all its maximal simplices have dimension $n$.

\begin{lemma}\label{dim2}
Let $K$ be a homogeneous simplicial complex of dimension 2 with no free faces and let $S$ be a surface. Suppose that $\chi(K) = \chi(S)$. Then, $K$ has at least $\rho(S)$ vertices.
\end{lemma}

\begin{proof}
Call $\alpha_0$, $\alpha_1$, $\alpha_2$ the number of vertices, edges and triangles of $K$ respectively. We know that
\[
	\chi(S) = \chi(K) = \alpha_0 - \alpha_1 + \alpha_2.
\]
Since $K$ has no free faces and no maximal edges, $3\alpha_2 \geq 2\alpha_1$. On the other hand, since $K$ is a simplicial complex, it has at most ${{\alpha_0}\choose{2}}$ edges. Then
\[
	6\chi(S) = 6\chi(K) \geq 6\alpha_0 - \alpha_0(\alpha_0 - 1).
\]
The minimum integer that satisfies this inequality is precisely $\rho(S)$, and therefore $\alpha_0 \geq \rho(S)$.
\end{proof}

\begin{rem}\label{property}
If $S$ is a surface, it is known that its cohomology ring $H^{*}(S;\FF_2)$, with coefficients in $\FF_2$, satisfies the following property:
\[
\text{for every } \alpha \neq 0 \text{ in } H^{1}(S;\FF_2) \text{, there exists } \beta \in H^{1}(S;\FF_2) \text{ such that } \alpha \cup \beta \neq 0 \text{ in } H^{2}(S;\FF_2). \label{eq:cohomprop} \tag{A}
\]
\end{rem}

Hereafter, the coefficient ring for homology and cohomology groups will be understood to be $\FF_2$. 

\begin{coro}\label{corodim2}
Let $K$ be a simplicial complex of dimension 2 and let $S$ be a surface. Suppose that $\chi(K) = \chi(S)$. If additionally $H^{*}(K)$ satisfies property \eqref{eq:cohomprop} of Remark \ref{property}, then $K$ has at least $\rho(S)$ vertices.
\end{coro}

\begin{proof}
Collapse every free face of $K$ to get a subcomplex $L$ with no free faces. If $L$ is homogeneous of dimension $2$, we can apply Lemma \ref{dim2} and we are done. We can assume  then that $L$ has maximal edges. We will show that it is always possible to replace $L$ by a homotopy equivalent simplicial complex with less vertices and less maximal edges.  Let $\sigma=\{a,b\}$ be a maximal edge. Suppose that there is a simple path $P$ joining $a$ to $b$ in $L \setminus \sigma$ (the simplicial complex obtained from $L$ by removing the edge $\sigma$). Consider the quotient $L / P$, which is homotopy equivalent to a space of the form $T \vee S^1$ for a suitable space $T$. Since on the other hand $L / P$ is homotopy equivalent to $L$, we have  $L \simeq T \vee S^1$. This is a contradiction because the cohomology ring $H^{*}(L)$ satisfies property $\eqref{eq:cohomprop}$ while $H^{*}(T \vee S^1) \equiv H^{*}(T) \times H^{*}(S^1)$ does not. Then there is no path between $a$ and $b$ in $L \setminus \sigma$ and therefore the quotient $L / \sigma$ has a natural simplicial structure and is homotopy equivalent to $L$. Thus, if we replace $L$ by $L / \sigma$ we obtain a simplicial complex homotopy equivalent to $K$ with less maximal edges than $L$. By iterating this procedure we end up with a homogeneous simplicial complex $\hat K$ of dimension $2$ homotopy equivalent to $K$ with no more vertices than $K$. By Lemma \ref{dim2} applied to $\hat K$, we deduce that $\hat K$ (and in consequence $K$) has at least $\rho(S)$ vertices.
\end{proof}

Suppose now that $K$ is a simplicial complex homotopy equivalent to a surface $S$ and consider its 2-skeleton $K^{(2)}$. It is clear that $\chi\left(K^{(2)}\right) \geq \chi(S)$. Moreover, the cohomology ring $H^{*}\left(K^{(2)}\right)$ satisfies property \eqref{eq:cohomprop} (see Lemma \ref{propstar} below). To prove the main result, we will find a subcomplex $T$ of the 2-skeleton of $K$ such that $H^{*}(T)$ has property \eqref{eq:cohomprop} and $\chi(T) = \chi(S)$ by carefully removing some of the 2-simplices of $K^{(2)}$.

Given a simplicial complex $K$, we will denote by $(C_{*}(K),d_{*})$ its simplicial chain complex with coefficients in $\FF_2$. Whenever is necessary to avoid confusions, we will add a superscript $K$ to the differential (as in $d_n^{K}$).

 In what follows, $K$ will be a simplicial complex homotopy equivalent to a surface $S$ and $Z \in C_2(K)$ will be a fixed 2-chain that generates $H_{2}(K)$. If $L\leq K$ is a subcomplex, we denote by $\dim H_2(L)$ the dimension of $H_2(L)$ as a vector space over $\FF_2$.

\begin{lemma}\label{propstar}
Let $L \leq K$ be a subcomplex such that the inclusion $i:L \hookrightarrow K$ induces isomorphisms $i_{*}:H_{n}(L) \to H_{n}(K)$ for $n < 2$ and an epimorphism for $n = 2$. Then, the cohomology ring $H^{*}(L)$ satisfies property \eqref{eq:cohomprop}.
\end{lemma}

\begin{proof}
Since $i_{*}:H_2(L) \to H_2(K)$ is an epimorphism, we can find a class $[C] \in H_2(L)$ with $i_{*}([C]) = [Z]$. In order to see that the cohomology ring of $L$ satisfies property \eqref{eq:cohomprop}, take $\alpha \in H^{1}\left(L\right)$ a nonzero element. Then $\alpha = i^{*}\left(\hat \alpha\right)$ with $\hat \alpha$ a nonzero element of $H^1(K)$. Since $K$ is homotopy equivalent to a surface, there exists $\hat \beta \in H^{1}(K)$ such that $(\hat\alpha \cup \hat\beta)([Z]) = 1$ (we identify $H^{2}(K)$ with $\Hom(H_{2}(K),\FF_2)$ by the universal coefficient theorem for cohomology). Then,
\[
(i^{*}( \hat \alpha ) \cup i^{*} ( \hat \beta) )[C] = i^{*}( \hat \alpha \cup \hat \beta)[C] = (\hat \alpha \cup \hat \beta)i_{*}([C]) = 
(\hat \alpha \cup \hat \beta)[Z] = 1,
\] 
which proves that $\alpha\cup\beta$ is nonzero.
\end{proof}

\begin{prop}\label{induc}
Let $L \leq K^{(2)}$ be a subcomplex satisfying the hypotheses of Lemma \ref{propstar}. If $\dim H_2(L) > 1$, there is a 2-simplex $\sigma \in L$ such that the inclusion $j:L \setminus \sigma \hookrightarrow K$ also satisfies the hypotheses of Lemma \ref{propstar}, i.e. it induces isomorphisms  $j_{*}:H_{n}(L \setminus \sigma) \to H_{n}(K)$ for $n < 2$ and an epimorphism for $n = 2$.
\end{prop}

\begin{proof}
By the commutativity of the following diagram
\[
	\begin{tikzcd}
	0 \ar{r} & C_2(L) \ar{r}\ar{d} & C_1(L) \ar{r}\ar{d} & C_0(L) \ar{r}\ar{d} & 0 \\
      & C_2(K) \ar{r} & C_1(K) \ar{r} & C_0(K) \ar{r} & 0,
	\end{tikzcd}
\]
where the vertical arrows are the inclusions induced by $i:L \hookrightarrow K$, we can identify $\ker d_2^{L}$ with a subspace of $\ker d_2^{K}$. Let $C \in \ker d_2^{L} \equiv H_2(L)$ such that $i_{*}[C] = [Z]$. Since $\dim H_2(L) > 1$ and $\ker d_2^{K}$ admits a decomposition of the form $\ker d_2^{K} = \langle Z \rangle \oplus \im d_3^{K}$, there is a nonzero 2-cycle $B \in \ker d_2^{L} \cap \im d_3^{K}$. Let $\sigma$ be a 2-simplex in the support of $B$. By the choice of $\sigma$, the inclusion induces the zero morphism $ H_1(\partial \sigma) \to H_1(L\setminus{\sigma})$. Therefore the inclusion $L \setminus \sigma \hookrightarrow L$ induces isomorphisms $H_n(L \setminus \sigma) \to H_n(L)$ for $n < 2$. It remains to verify that the inclusion $j: L\setminus \sigma \hookrightarrow K$ induces an epimorphism in $H_2$. If $\sigma \not\in \supp C$, clearly $j_{*}[C] = [Z]$. If  $\sigma \in \supp C$, the 2-chain $C + B$ is a well defined 2-cycle in $L \setminus \sigma$ and $j_{*}[C+B] = [i_{*}[C] + B] = [Z]$. Hence, in both cases $j_{*}: H_2(L \setminus \sigma) \to H_2(K)$ is an epimorphism.
\end{proof}

By applying iteratively Proposition \ref{induc} to $L = K^{(2)}$, we end up with a subcomplex $T \leq K^{(2)}$ such that the inclusion $i:T \hookrightarrow K$ induces isomorphisms in all homology groups and, in particular $H^*(T)$ satisfies property \eqref{eq:cohomprop}. The proof of the main theorem follows immediately from this fact.

\begin{theo}\label{main}
Let $K$ be a simplicial complex homotopy equivalent to a surface $S$. Then $K$ has at least $\rho(S)$ vertices. In particular, if $S\neq M_2, N_2, N_3$ then $\ct(S)=\delta(S)$.
\end{theo}

\begin{proof}
By Proposition \ref{induc}  there exists a subcomplex $T \leq K^{(2)}$ with $\chi(T) = \chi(S)$ and such that its cohomology ring satisfies property \eqref{eq:cohomprop}. Since the number of vertices of $T$ is less than or equal to the number of vertices of $K$, by Corollary \ref{corodim2}, $K$ has at least $\rho(S)$ vertices.
\end{proof}

\section{The exceptional cases}

In this section we analyze the covering type of the exceptional surfaces $M_2, N_2$ and $N_3$. By Theorem \ref{main}, for $S = M_2,N_2,N_3$ the covering type of $S$ is between $\delta(S)-1$ and $\delta(S)$. We complete the computation of their covering types by showing that $\ct(N_2) = \delta(N_2)$, $\ct(N_3) = \delta(N_3)$ and exhibiting a simplicial complex on $\delta(M_2)-1$ vertices homotopy equivalent to $M_2$.

Suppose we have a simplicial complex homotopy equivalent to a surface $S$. By Theorem \ref{main} and Proposition \ref{induc}, we can obtain a simplicial complex of dimension $2$ with the homology of $S$ and whose cohomology ring possesses property \eqref{eq:cohomprop}. The next result states, roughly, that such a complex is close to being a surface.

Recall that the link $\lk_K(v)$ of a vertex $v$ of $K$ is the subcomplex of $K$ of simplices $\sigma$ disjoint with $v$ such that $\sigma \cup \{v\}\in K$.

\begin{prop}\label{homeosurf}
Let $K$ be a pure 2-complex such that each edge of $K$ is the face of exactly 2 triangles and let $S$ be a surface. Suppose there is a continuous map $f:K \to S$ that induces isomorphisms in all homology groups. Then $K$ is homeomorphic to $S$.
\end{prop}

\begin{proof}
Each strongly connected component $C$ of $K$ is a pseudosurface without boundary, so that $H^2(C) \equiv \FF_2$. Since two different strongly connected components of $K$ may intersect only at vertices, the dimension of $H^2(K)$ coincides with the number of its strongly connected components. Hence, $K$ is strongly connected. Since every edge of $K$ is the face of exactly two 2-simplices, the link of every vertex is a 2-regular graph and hence, homeomorphic to a disjoint union of some copies of $S^{1}$. Suppose there is a vertex such that its link has $k > 1$ connected components. It is easy to see that in this case $K \simeq L \vee_{i=1}^{k-1} S^{1}$ for a suitable 2-complex $L$. Since the map $f:K \to S$ induces an isomorphism in cohomology, $H^*(K)$ satisfies property \eqref{eq:cohomprop} while $H^*(L \vee_{i=1}^{k-1} S^{1})$ does not. Hence, the link of every vertex of $K$ is homeomorphic to $S^{1}$. This implies that $K$ is a surface. By the classification of surfaces, $K$ is homeomorphic to $S$. 
\end{proof}

We are now ready to compute the covering type of $N_2$ and $N_3$.

\begin{prop}
Let $S = N_2$ or $N_3$. Then $\ct(S) = \delta(S)$.
\end{prop}
\begin{proof}
Let $K$ be a simplicial complex on $\ct(S)$ vertices such that there is a homotopy equivalence $f:K\to S$. By Proposition \ref{induc}, there is a subcomplex $L \leq K^{(2)}$ such that the map $L \hookrightarrow K \xrightarrow[]{f} S$ induces an isomorphism in homology. By collapsing the free faces of $L$ and proceeding as in Corollary \ref{corodim2} we can assume further that $L$ is pure and has no free faces. Let $\alpha_0$, $\alpha_1$, $\alpha_2$ denote the number of vertices, edges and triangles of $L$ respectively. By a straightforward computation of the Euler characteristic,
\[ 3(\alpha_0 - \chi(S)) \leq \alpha_1 \leq { {\alpha_0} \choose {2} }.\]
This inequality immediately implies that $\ct(N_2) = \delta(N_2) = 8$. Indeed, if  $\ct(N_2) = 7$, we would have $\alpha_1 = 21$ and $\alpha_2 = 14$. Since $3\alpha_2 = 2 \alpha_1$ every edge of $L$ is a face of exactly two triangles. By Proposition \ref{homeosurf}, $L$ is homeomorphic to $N_2$, a contradiction. Hence $\ct(N_2) = 8$.

Consider now the case $S = N_3$. Suppose that $\ct(N_3) = \delta(N_3) - 1 = 8$. Then $\alpha_1 = 27$ or $28$. As before, if $\alpha_1 = 27$ then $L$ should be homeomorphic to $N_3$ and this is impossible.  If $\alpha_1 = 28$, every edge of $L$ is the face of 2 triangles except for one edge that is contained in 3 triangles. Let $v \in L$ be a vertex of this edge. The link of $v$ is a graph in which exactly one vertex has degree 3 and every other vertex is of degree 2. Since the sum of the degrees of all vertices should be even, this is a contradiction. Hence $\ct(N_3) = \delta(S) = 9$ and the proof is complete.
\end{proof}

\begin{lemma}
Let $K$ be a simplicial complex of dimension $2$. Suppose there are two vertices $v, v' \in K$ not connected by an edge and with disjoint links. Suppose further that there are vertices $w \in \lk_{K} (v)$, $w' \in \lk_{K}(v')$ connected by an edge in $K$. Then, the complex $(K / (v \sim v')) \cup \sigma$ has a natural simplicial structure and is homotopy equivalent to $K$, where $\sigma$ is the 2-simplex $\{[v],w,w'\}$  (here $[v]$ denotes the class of the vertex $v$ in the quotient complex).
\end{lemma}

\begin{proof}
First note that the quotient complex $L = K / (v \sim v')$ inherits a simplicial structure. Since $v$ and $v'$ do not form an edge in $K$, the edges of the $1$-skeleton of $L$ connect different vertices. Moreover, since the links of $v$ and $v'$ are disjoint in $K$, there is at most one edge between two given vertices in the $1$-skeleton of $L$. Finally, it is clear that the 2-simplices of $L$ may intersect only on a vertex or an edge. It is not difficult to verify that if we add the 2-simplex $\sigma=\{[v],w,w'\}$ to $L$ we get a simplicial complex homotopy equivalent to $K$.
\end{proof}

\begin{prop}
There is a simplicial complex of dimension 2 with 9 vertices homotopy equivalent to $M_2$. In particular, $\ct(M_2) = 9$.
\end{prop}

\begin{proof}
Consider the minimal triangulation $T$ of $M_2$ described in \cite{JR} (see also \cite{KW}). The 1-skeleton of the simplicial complex $T$ is a graph on 10 vertices such that:
\begin{itemize}
	\item There are two vertices $v,v'$ of degree $4$ that do not form an edge and with disjoint links.
    \item The subgraph induced by the remaining 8 vertices is the complete graph $K_8$.
\end{itemize}
Choose vertices $w \in \lk_{T}(v)$, $w' \in \lk_{T}(v')$. By the previous Lemma, the simplicial complex $T / (v \sim v') \cup \{[v],w,w'\}$ is homotopy equivalent to $T \equiv M_2$.
\end{proof}

\end{document}